\documentclass[12pt]{article}
\usepackage{stmaryrd}
\usepackage{amsmath}
\usepackage{graphicx}
\usepackage{color}
\usepackage[title]{appendix}
\usepackage{amsfonts,amscd,amssymb, mathtools,mathrsfs, dsfont}

\usepackage{pgfplots}
\usepackage{float}

\usepackage{url}
\usepackage{hyperref}

\usepackage{array}
\hypersetup{
colorlinks=true, 
breaklinks=true, 
urlcolor= blue, 
linkcolor= red, 
citecolor=purple, 
pdftitle={}, 
pdfauthor={}, 
pdfsubject={} 
}

\tikzset{elegant/.style={smooth,thick,samples=50,cyan}}
\tikzset{eaxis/.style={->,>=stealth}}

\newif\ifblog
\newif\iftex
\blogfalse
\textrue

\newcommand{\thmref}[1]{Theorem~{\rm \ref{#1}}}
\newcommand{\lemref}[1]{Lemma~{\rm \ref{#1}}}

\newcommand{\propref}[1]{Proposition~{\rm \ref{#1}}}

\newcommand{\figref}[1]{Figure~{\rm \ref{#1}}}

\def\wu{\widehat{u}}
\def\wU{\widehat{U}}
\def\d{{\rm d}}
\def\P{{\mathbb P}}
\def\E{{\mathbb E}}

\def\Z{{\mathbb Z}}

\def\p{{\partial}}

\newcommand{\cF}{{\cal F}}
\newcommand{\ep}{\varepsilon}
\newcommand{\al}{\alpha}

\newcommand{\nd}{\noindent}

\newcommand{\la}{\lambda}

\newcommand{\si}{\sigma}

\newcommand{\rt}{\rightarrow}

\newtheorem{theorem}{Theorem}[section]
\newtheorem{lemma}[theorem]{Lemma}

\newtheorem{proposition}[theorem]{Proposition}

\newenvironment{proof}{\noindent {\sc Proof:}}{\strut\hfill $\Box$} 

\newcommand{\LL}{{\cal L}}
\newcommand{\TT}{{\cal T}}

\newcommand{\dd}{\operatorname{d}\!}
\newcommand{\dt}{\operatorname{d}\! t}

\renewcommand{\geq}{\geqslant}
\renewcommand{\leq}{\leqslant}

\title{A free boundary problem arising from a multi-state regime-switching stock trading model}

\author{Chonghu Guan\thanks{School of Mathematics, Jiaying University, Meizhou 514015, Guangdong,
China. This author is partially supported by NNSF of China (No. 11901244), NSF of Guangdong Province of China (No. 2016A030307008). Email: \url{316346917@qq.com}.} \and Jing Peng\thanks{Department of Applied Mathematics, The Hong Kong Polytechnic University, Hong Kong, China. Email: \url{jing.peng@connect.polyu.hk}.}\and Zuo Quan Xu\thanks{Department of Applied Mathematics, The Hong Kong Polytechnic University, Hong Kong, China. This author is partially supported by NSFC (No. 11971409), and Hong Kong
GRF (No. 15204216 and No. 15202817). Email: \url{maxu@polyu.edu.hk}.
}}
\date{}

\begin{document}
\maketitle

\begin{abstract}
In this paper, we study a free boundary problem, which arises from an optimal trading problem of a stock that is driven by a uncertain market status process. The free boundary problem is a variational inequality system of three functions with a degenerate operator. The main contribution of this paper is that we not only prove all the four switching free boundaries are no-overlapping, monotonic and $C^{\infty}$-smooth, but also completely determine their relative localities and provide the optimal trading strategies for the stock trading problem.

\bigskip
\nd {\bf Keywords.} free boundary problem; system of parabolic variational inequalities; regime-switching; stock trading

\bigskip
\nd {\bf Mathematics Subject Classification.}
35R35; 
35K87; 
91B70; 
91B60.


\end{abstract}


\setlength{\baselineskip}{0.25in}

\section{Introduction}

This paper considers a free boundary problem arising from a stock trading model. The stock price is driven by a two-state market status process which is unobservable to the trader. We leave the details of the financial and stochastic background of the problem to the interested readers in the Appendix.

The model can be reduced to finding a triple of value functions $(v_0(p,t),\;v_1(p,t),\;v_{-1}(p,t))$ that satisfies the following variational inequality (VI) system
\begin{align}\label{v_eq}
\left\{
\begin{array}{ll}
\min\Big\{\p_t v_0-\LL v_0,\;v_0-v_1+(1+K),\;v_0-v_{-1}-(1-K)\Big\}=0,\\[2mm]
\min\Big\{\p_t v_1-\LL v_1,\;v_1-v_0-(1-K)\Big\}=0,\\[2mm]
\min\Big\{\p_t v_{-1}-\LL v_{-1},\;v_{-1}-v_0+(1+K)\Big\}=0, \qquad (p,t)\in\Omega,
\end{array}
\right.
\end{align}
with the initial conditions
\begin{align}\label{v_0}
\left\{
\begin{array}{ll}
v_0(p,0)=0,\\[2mm]
v_1(p,0)=1-K,\\[2mm]
v_{-1}(p,0)=-(1+K),\qquad 0<p<1,
\end{array}
\right.
\end{align}
where
\[\Omega=(0,1)\times(0,T],\]
and the operator $\LL$ is defined by
\begin{align*}
\LL &=\frac{1}{2}\left(\frac{(\mu_1-\mu_2)p(1-p)}{\si}\right)^2\p_{pp}\\[2mm]
&\quad\;+\big(-(\la_1+\la_2)p+\la_2+(\mu_1-\mu_2)p(1-p)\big)\p_p+(\mu_1-\mu_2)p+\mu_2-\rho.
\end{align*}
These parameters, including the percentage of transaction fee $K$, the expected return rates of the stock $\mu_1$ and $\mu_{2}$ in the different market statues, the volatility of the stock $\si$, the discount rate of the investor $\rho$, and the coefficients of transition probabilities of the market states $\la_1$ and $\la_2$, are all constants and satisfy
\[0<K<1, \quad \mu_1>\rho>\mu_{2},\quad \si,\; \la_1,\; \la_2>0.\]

Note that the operator $\LL$ is degenerate on the boundaries $p = 0$ and $p=1$.
According to Fichera's theorem (see \cite{OR73}), we must not put the boundary condition on these two boundaries.

Partial differental equation (PDE) technologies are widely used in the economic literature to study similar regime-switching models. For instance, Yi \cite{Yi08} considers
the pricing problem of American put option with regime-switching volatility and its related excising region; Khaliq, Kleefeld and Liu \cite{KK13} studies the numerical solution of a class of complex PDE systems about American option with multiple states regime-switching.
Dai and Zhang et al. \cite{DZ11} not only give theoretical analysis on the variational inequalities, but also provide numerical simulations.

Dai and Zhang et al \cite{DZ10} consider the optimal stock trading rule with two net positions: the flat position (no stock holding) and the long position (holding one share of stock), the model can be expressed by a variational inequality system of the corresponding two value functions. As shown in that paper, the two variational inequalities can be reduced into one double obstacle problem on the difference of the two value functions, and then the one-dimensional problem and the properties of its free boundary can be obtained by well-known results on the double obstacle problem in the PDE literature.

Our model is an evolutionary problem of Dai and Zhang et al \cite{DZ10}, in which a third financially meaningful position - short position - is allowed. But it is not a trivial extension in the following sense. Because the stock are allowed to be short in our model, there is one extra (short) position than that of \cite{DZ10}. As a consequence, the Hamilton-Jacobi-Bellman equation (HJB) system involves three variational inequalities and three value functions, which could not be amalgamated into a single one as previous work \cite{DX10,DY09,DZ10} do, so  it calls for completely new technologies to deal with.
\par
Our model is also similar to Ngo and Pham \cite{NP16}. They consider the problem of determining the optimal cut-off of the pair trading rule for a three-state regime-switching model. However their model is infinity time horizon, so there is not time variable and the system is an ordinary differential equation system for which standard smooth-fitting technique can be applied and a closed-form solution is often available. By contrast, we consider a finite time model, so the system becomes a PDE system for which the existence and smoothness of the solution is much harder to establish.
\par
In this paper, we prove the VI system \eqref{v_eq} with the initial conditions \eqref{v_0} has four switching free boundaries and they are no-overlapping, monotonic and $C^{\infty}$-smooth, we also completely determine their relative localities and provide the optimal trading strategies for the stock trading problem.
The main contribution on mathematical method of this paper is that, the properties of the free boundaries of the variational inequality system with coupling appears in obstacle constraints are discussed and proved for the first time.
\par
The rest of the paper is arranged as follows. In Section 2, we first construct a penalty approximation system and obtain some estimations of its solution, then completely solve the VI system \eqref{v_eq} by a limit argument.
Section 3 is devoted to the study of the properties of the four switching free boundaries of the VI system \eqref{v_eq}.
In the appendix, we give the financial and stochastic background of the problem.



\section{Existence and uniqueness by penalty approximation method}\label{sec:Penalty}
\setcounter{equation}{0}

In this section, we use approximation method to study the problem \eqref{v_eq}. We will show the system has a unique solution.
\par
For this, we first construct an approximation equation system for the problem \eqref{v_eq}. For sufficiently small $\ep>0$, consider
\begin{align}\label{ve_eq}
\left\{
\begin{array}{ll}
\p_t v^\ep_0-\LL v^\ep_0+\beta_\ep(v^\ep_0-v^\ep_1+1+K)+\beta_\ep(v^\ep_0-v^\ep_{-1}-(1-K))=0,\\[2mm]
\p_t v^\ep_1-\LL v^\ep_1+\beta_\ep(v^\ep_1-v^\ep_0-(1-K))=0,\\[2mm]
\p_t v^\ep_{-1}-\LL v^\ep_{-1}+\beta_\ep(v^\ep_{-1}-v^\ep_0+(1+K))=0, \qquad (p,t)\in\Omega,
\end{array}
\right.
\end{align}
with the initial conditions
\begin{align}\label{ve_0}
\left\{
\begin{array}{ll}
v^\ep_0(p,0)=0,\\[2mm]
v^\ep_1(p,0)=1-K,\\[2mm]
v^\ep_{-1}(p,0)=-(1+K),\qquad 0<p<1;
\end{array}
\right.
\end{align}
where $\beta _{\ep }(\cdot)$ is any penalty function satisfying the following properties
\begin{eqnarray*}
&&\beta _{\ep }(\cdot)\in C^{2}(-\infty,+\infty ), \quad\beta _{\ep}(0)=-c_0, \quad\beta _{\ep}(\ep)=-c_1, \quad\beta _{\ep}(x)=0\; \text{ for } x\geq 2\ep, \\[2mm]
&& \beta _{\ep
}(\cdot)\leq 0, \quad \beta _{\ep }^{\prime }(\cdot)\geq 0,\quad \beta _{\ep
}^{\prime \prime }(\cdot)\leq 0, \quad\lim\limits_{\ep \rightarrow 0} \beta _{\ep }(x)=\left\{
\begin{array}{ll}
0, & x>0,\vspace{2mm} \\
-\infty, & x<0,
\end{array}
\right.
\end{eqnarray*}%
and $c_0>c_1>0$ are two constants that are independent of $\ep$ and will be determined latter. \figref{fig1} demonstrates such a penalty function $\beta _{\ep }(\cdot)$.
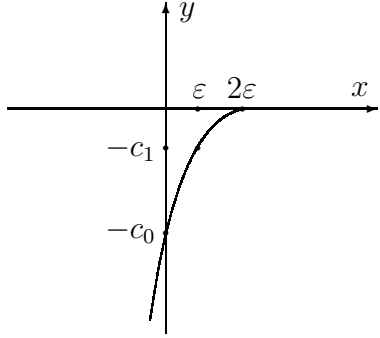
\begin{figure}[H]
\begin{center}
\begin{picture}(250,140)
\put(50,80){\vector(1,0){140}}
\put(110,-5){\vector(0,1){125}}
\qbezier(104,0)(115,80)(142,80)
\put(180,85){$x$}
\put(115,115){$y$}
\put(120,84){$\ep$}
\put(133,84){$2\ep$}
\put(87,31){$-c_0$}
\put(87,62){$-c_1$}
\put(110,33){\circle*{2}}
\put(122,80){\circle*{2}}
\put(122,65){\circle*{2}}
\put(110,65){\circle*{2}}
\put(139,80){\circle*{2}}
\end{picture} \caption{A penalty function $\beta_\ep (\cdot)$.} \label{fig1}
\end{center}
\end{figure}

We first, using the technique of Dai and Yi \cite{DY09}, reduce the 3-dimension problem \eqref{ve_eq} to a 2-dimension problem. Let
\[
u^\ep_1=v^\ep_0-v^\ep_1,\quad u^\ep_{-1}=v^\ep_0-v^\ep_{-1}.
\]
By \eqref{ve_eq} and \eqref{ve_0}, we have
\begin{align}\label{ue_eq}
\left\{
\begin{array}{ll}
\p_t u^\ep_1-\LL u^\ep_1+\beta_\ep(u^\ep_1+1+K)-\beta_\ep(-u^\ep_1-(1-K))+\beta_\ep(u^\ep_{-1}-(1-K))=0,\\[2mm]
\p_t u^\ep_{-1}-\LL u^\ep_{-1}+\beta_\ep(u^\ep_{-1}-(1-K))-\beta_\ep(-u^\ep_{-1}+(1+K))+\beta_\ep(u^\ep_1+1+K)=0,
\end{array}
\right.
\end{align}
and
\begin{align}\label{ue_0}
\left\{
\begin{array}{ll}
u^\ep_1(p,0)=-(1-K),\\[2mm]
u^\ep_{-1}(p,0)=1+K,\qquad 0<p<1.
\end{array}
\right.
\end{align}
Since the operator $\LL$ is degenerated on the boundaries $p = 0$ and $p=1$, we first consider the problem \eqref{ue_eq} in the bounded domain
\[\Omega_\ep:=(\ep,1-\ep)\times(0,T],\]
with the initial conditions \eqref{ue_0} and the additional boundary conditions
\begin{align}\label{ue_bc}
\left\{
\begin{array}{ll}
\p_p u^\ep_1(\ep,t)=\p_p u^\ep_1(1-\ep,t)=0,\\[2mm]
\p_p u^\ep_{-1}(\ep,t)=\p_p u^\ep_{-1}(1-\ep,t)=0,\qquad 0\leq t<T.
\end{array}
\right.
\end{align}

We now provide the following existence and uniqueness results for the approximation problem \eqref{ue_eq}.
\begin{lemma}\label{thm:ue}
The system \eqref{ue_eq}, restricted to the domain $\Omega_\ep$, with the initial conditions \eqref{ue_0} and boundary conditions \eqref{ue_bc}, admits a unique solution $(u^\ep_1,u^\ep_{-1})\in C^{2,1}(\overline{\Omega}_\ep)\times C^{2,1}(\overline{\Omega}_\ep)$. Moreover, the solution satisfies
\begin{align}\label{ue1}
-(1+K)+\ep\leq &u^\ep_1\leq -(1-K),\\
1-K+\ep\leq &u^\ep_{-1}\leq 1+K, \label{ue-1}
\end{align}
in $\overline{\Omega}_\ep$.
\end{lemma}
\begin{proof}
The existence and uniqueness can be proved by Schauder fix point theorem and the comparison principle; as the process is standard, we leave it to the interested readers.
\par
We come to prove the first inequality in \eqref{ue1}. Denote
\[
\phi=-(1+K)+\ep,
\]
and choose $c_1=(\mu_1-\mu_{2})(1+K)+1>0$. Then by simple calculations,
\begin{align*}
&\quad\;\p_t \phi-\LL \phi+\beta_\ep(\phi+1+K)-\beta_\ep(-\phi-(1-K))+\beta_\ep(u^\ep_{-1}-(1-K))\\
&=[(\mu_1-\mu_2)p+\mu_2-\rho][1+K-\ep]+\beta_\ep(\ep)-\beta_\ep(2K-\ep)+\beta_\ep(u^\ep_{-1}-(1-K))\\
&=[(\mu_1-\mu_2)p+\mu_2-\rho][1+K-\ep]-c_1+\beta_\ep(u^\ep_{-1}-(1-K))\\
&\leq [(\mu_1-\mu_2)+\mu_2-\mu_2][1+K-\ep]-((\mu_1-\mu_{2})(1+K)+1)\\
&< 0,
\end{align*}
so by \eqref{ue_eq},
\begin{align*}
&\quad\;\p_t \phi-\LL \phi+\beta_\ep(\phi+1+K)-\beta_\ep(-\phi-(1-K))\\
&<\p_t u^\ep_1-\LL u^\ep_1+\beta_\ep(u^\ep_1+1+K)-\beta_\ep(-u^\ep_1-(1-K)).
\end{align*}
By the comparison principle, we have $u^\ep_1\geq \phi =-(1+K)+\ep$, proving the first inequality in \eqref{ue1}. Similarly, using $c_1\geq (\rho-\mu_2)(1-K+\ep)$, we can prove $ 1-K+\ep \leq u^\ep_{-1}$ in \eqref{ue-1}.
\par
Now, we prove the second inequality in \eqref{ue1}.
Denote
\[
\Phi=-(1-K),
\]
and set $c_0=c_1+(\mu_{1}-\mu_2)(1+K)+1>c_{1}$. Then
\begin{align*}
&\quad\;\p_t \Phi-\LL \Phi+\beta_\ep(\Phi+1+K)-\beta_\ep(-\Phi-(1-K))+\beta_\ep(u^\ep_{-1}-(1-K))\\
&=[(\mu_1-\mu_2)p+\mu_2-\rho](1-K)+\beta_\ep(2K)-\beta_\ep(0)+\beta_\ep(u^\ep_{-1}-(1-K))\\
&=[(\mu_1-\mu_2)p+\mu_2-\rho](1-K)+c_0+\beta_\ep(u^\ep_{-1}-(1-K))\\
&\geq (\mu_2-\rho)(1-K)+c_0+\beta_\ep(1-K+\ep-(1-K))\\
&= (\mu_2-\rho)(1-K)+c_0-c_1\\
&> 0,
\end{align*}
where the inequality is due to $\beta'\geq 0$ and the first inequality  in \eqref{ue-1}.
It follows from \eqref{ue_eq} that
\begin{align*}
&\quad\;\p_t \Phi-\LL \Phi+\beta_\ep(\Phi+1+K)-\beta_\ep(-\Phi-(1-K))\\
&>\p_t u^\ep_1-\LL u^\ep_1+\beta_\ep(u^\ep_1+1+K)-\beta_\ep(-u^\ep_1-(1-K)).
\end{align*}
By the comparison principle, we have $u^\ep_1\leq \Phi=-(1-K)$, giving the second inequality in \eqref{ue1}.
Similarly, using $c_0\geq c_1+ (\mu_1-\rho)(1+K)$, we can prove $ u^\ep_{-1}\leq 1+K$ in \eqref{ue-1}.
\end{proof}

\begin{lemma}\label{pro:uetp}
Let $(u^\ep_1,u^\ep_{-1})$ be given in \lemref{thm:ue}, we have
\begin{align}\label{uep}
&\p_pu^\ep_1\leq 0,\quad \p_p u^\ep_{-1}\geq 0.
\end{align}
\end{lemma}
\begin{proof}
Denote $w_i=\p_p u^\ep_i$, for $i=1$, $-1$. After differentiating \eqref{ue_eq} w.r.t. $p$ and we get
\begin{align}\label{w_eq}
\left\{
\begin{array}{ll}
\p_t w_1-\TT w_1-(\mu_1-\mu_2)u^\ep_1+\beta_\ep'(u^\ep_1+1+K)w_1\\[2mm]
\qquad\qquad +\beta_\ep'(-u^\ep_1-(1-K))w_1+\beta_\ep'(u^\ep_{-1}-(1-K))w_{-1}=0,\\[2mm]
\p_t w_{-1}-\TT w_{-1}-(\mu_1-\mu_2)u^\ep_{-1}+\beta_\ep'(u^\ep_{-1}-(1-K))w_{-1}\\[2mm]
\qquad\qquad\quad +\beta_\ep'(-u^\ep_{-1}+(1+K))w_{-1}+\beta_\ep'(u^\ep_1+1+K)w_1=0,
\end{array}
\right.
\end{align}
where the operator $\TT$ is defined by
\begin{align*}
\TT&:=\LL+\frac{(\mu_1-\mu_2)^2}{\si^2}p(1-p)(1-2p)\p_p+[-(\la_1+\la_2)+(\mu_1-\mu_2)(1-2p)]\\
&=\frac{1}{2}\left(\frac{(\mu_1-\mu_2)p(1-p)}{\si}\right)^2\p_{pp}\\[2mm]
&\quad\;+\left(-(\la_1+\la_2)p+\la_2+(\mu_1-\mu_2)p(1-p)+\frac{(\mu_1-\mu_2)^2}{\si^2}p(1-p)(1-2p)\right)\p_p\\
&\quad\;+\mu_2-\rho-(\la_1+\la_2)+(\mu_1-\mu_2)(1-p).
\end{align*}

Define
\[
W_i=e^{-\la t}w_i,\quad i=1,-1,
\]
where $\la$ is a constant that will be determined latter. It sufficies to prove
\begin{align}\label{W<0}
W_1\leq 0,\quad W_{-1}\geq 0.
\end{align}
Using $u^\ep_1<0$, $u^\ep_{-1}>0$ and $\mu_1>\mu_2$, we get from \eqref{w_eq} that
\begin{align}\label{W_eq}
\left\{
\begin{array}{l}
\p_t W_1-\TT W_1+\la W_1+\beta_\ep'(u^\ep_1+1+K)W_1\\[2mm]
\qquad\quad +\beta_\ep'(-u^\ep_1-(1-K))W_1+\beta_\ep'(u^\ep_{-1}-(1-K))W_{-1}<0,\\[2mm]
\p_t W_{-1}-\TT W_{-1}+\la W_{-1}+\beta_\ep'(u^\ep_{-1}-(1-K))W_{-1}\\[2mm]
\qquad\qquad +\beta_\ep'(-u^\ep_{-1}+(1+K))W_{-1}+\beta_\ep'(u^\ep_1+1+K)W_1>0.
\end{array}
\right.
\end{align}
Clearly, there exist $(p_1,t_1)$, $(p_{-1},t_{-1})\in\overline{\Omega}_\ep$ such that
\[
W_1(p_1,t_1)=\max\limits_{(p,t)\in\overline{\Omega}_\ep}W_1(p,t),\quad W_{-1}(p_{-1},t_{-1})=\min\limits_{(p,t)\in\overline{\Omega}_\ep}W_{-1}(p,t).
\]
Suppose \eqref{W<0} was not true, then we would have \[\max\{W_1(p_1,t_1),-W_{-1}(p_{-1},t_{-1})\}>0.\]
Without loss of generality, we may assume that
\[W_1(p_1,t_1)\geq -W_{-1}(p_{-1},t_{-1})\quad \text{ and }\quad W_1(p_1,t_1)>0.\]
Notice $W_1=0$ on $\p_p \Omega_\ep$ (the parabolic boundary of $\Omega_\ep$), so $(p_1,t_1)$ is inside the domain $\Omega_\ep$ or at the upper boundary of $\Omega_\ep$, and thus,
\[
\p_p W_1(p_1,t_1)=0,\quad \p_{pp} W_1(p_1,t_1)\leq0,\quad \p_t W_1(p_1,t_1)\geq0.
\]
Choosing $\la >\mu_2-\rho-(\la_1+\la_2)+(\mu_1-\mu_2)(1-p)+1$, the above, after simple calculation, implies that
\begin{align}\label{Teq}
\p_t W_1-\TT W_1+\left(\la-\beta_\ep'(0)\right) W_1\bigg|_{(p_1,t_1)}>0.
\end{align}
One the other hand, using $W_1(p_1,t_1)>0$, and $\beta'\geq 0$, the first inequality in \eqref{W_eq} leads to
\begin{align*}
&\quad\;\p_t W_1-\TT W_1+\la W_1\bigg|_{(p_1,t_1)}\\
&<-\beta_\ep'(u^\ep_1+1+K)W_1-\beta_\ep'(-u^\ep_1-(1-K))W_1-\beta_\ep'(u^\ep_{-1}-(1-K))W_{-1}\bigg|_{(p_1,t_1)}\\
&\leq -\beta_\ep'(u^\ep_{-1}-(1-K))W_{-1}\bigg|_{(p_1,t_1)},
\end{align*}
noticing that $-W_{-1}(p_{1},t_{1})\leq -W_{-1}(p_{-1},t_{-1})\leq W_1(p_1,t_1)$, $\beta''\leq 0$, and by \eqref{ue-1} $u^\ep_{-1}>1-K$, the above is
\begin{align*}
&\leq \beta_\ep'(u^\ep_{-1}-(1-K))W_1 \bigg|_{(p_1,t_1)} \leq\beta_\ep'(0)W_1\bigg|_{(p_1,t_1)},
\end{align*}
contradicting to \eqref{Teq}. Therefore, \eqref{W<0} holds true and the claim is proved.
\end{proof}

\begin{lemma}\label{lem:ue1ue-1}
Let $(u^\ep_1,u^\ep_{-1})$ be given in \lemref{thm:ue}, we have
\begin{align}\label{ue1ue-1}
&u^\ep_1 +u^\ep_{-1}\geq0.
\end{align}
\end{lemma}
\begin{proof}
Denote $U=u^\ep_1 +u^\ep_{-1}$.
Adding up the two equations in \eqref{ue_eq}, we have
\begin{align*}
&\p_t U-\LL U+\big[\beta_\ep(u^\ep_1+1+K)-\beta_\ep(-u^\ep_{-1}+(1+K))\big]\\
+&\big[\beta_\ep(u^\ep_{-1}-(1-K))-\beta_\ep(-u^\ep_1-(1-K))\big]
+\beta_\ep(u^\ep_1+1+K)+\beta_\ep(u^\ep_{-1}-(1-K))=0.
\end{align*}
Using the mean value theorem, there exit $\xi(p,t)$ between $u^\ep_1+1+K$ and $-u^\ep_{-1}+1+K$, $\eta(p,t)$ between $u^\ep_{-1}-(1-K)$ and $-u^\ep_{-1}+(1+K)$, such that
\begin{align*}
\beta_\ep(u^\ep_1+1+K)-\beta_\ep(-u^\ep_{-1}+(1+K))=\beta_\ep'(\xi)U,\\ \beta_\ep(u^\ep_{-1}-(1-K))-\beta_\ep(-u^\ep_1-(1-K))=\beta_\ep'(\eta)U,
\end{align*}
thus,
\begin{align*}
\p_t U-\LL U+\beta_\ep'(\xi)U+\beta_\ep'(\eta)U=-\beta_\ep(u^\ep_1+1+K)-\beta_\ep(u^\ep_{-1}-(1-K))\geq 0,
\end{align*}
Using the maximum principle we get \eqref{ue1ue-1}.
\end{proof}

\begin{proposition}\label{lem:ve}
The system \eqref{ve_eq} with the initial conditions \eqref{ve_0} and the following boundary conditions
\begin{align}\label{ve_bc}
\left\{
\begin{array}{ll}
\p_p v^\ep_0(\ep,t)=\p_p v^\ep_0(1-\ep,t)=0,\\[2mm]
\p_p v^\ep_1(\ep,t)=\p_p v^\ep_1(1-\ep,t)=0,\\[2mm]
\p_p v^\ep_{-1}(\ep,t)=\p_p v^\ep_{-1}(1-\ep,t)=0,
\end{array}
\right.
\end{align}
has a unique solution $(v^\ep_0,v^\ep_1,v^\ep_{-1})$ with $v^\ep_i\in C^{2,1}(\overline{\Omega}_\ep)$, $i=0,1,-1$.
\end{proposition}
\begin{proof}
Let $(u^\ep_1,u^\ep_{-1})$ be given in \lemref{thm:ue}. Then the following linear problem of $v^\ep_0$,
\begin{align*}
\left\{
\begin{array}{ll}
\p_t v^\ep_0-\LL v^\ep_0+\beta_\ep(u^\ep_1+1+K)+\beta_\ep(u^\ep_{-1}-(1-K))=0, \\[2mm]
\p_p v^\ep_0(\ep,t)=0,\quad \p_p v^\ep_0(1-\ep,t)=0, \quad v^\ep_0(p,0)=0, \qquad (p,t)\in\Omega_{\ep},
\end{array}
\right.
\end{align*}
has a unique solution $v^\ep_0\in C^{2,1}(\overline{\Omega}_\ep)$. Let $v^\ep_1=v^\ep_0-u^\ep_1$ and $v^\ep_{-1}=v^\ep_0-u^\ep_{-1}$, then it is easy to verify that $(v^\ep_0,v^\ep_1,v^\ep_{-1})$ is a solution to the system \eqref{ve_eq} with the initial conditions \eqref{ve_0} and the boundary conditions \eqref{ve_bc}.The uniqueness can be proved easily by the comparison principle.
\end{proof}

Our first main result, which completely characterizes the solution of the system \eqref{v_eq}, is given below.
\begin{theorem}\label{thm:v}
The variational inequality system \eqref{v_eq} with the initial conditions \eqref{v_0}
has a unique solution $(v_0,v_1,v_{-1})$ such that $v_i\in C^{1+\al,\frac{1+\al}{2}}(\Omega\bigcup\{t=0\})\bigcap W^{2,1}_{q,\mathrm{loc}}(\Omega\bigcup\{t=0\})$, for any $q>3$, $i=0,1,-1$ and $\al=1-3/q$. Furthermore,
\begin{align}\label{u1}
-(1+K)\leq &v_0-v_1\leq -(1-K),\\\label{u-1}
1-K\leq &v_0-v_{-1}\leq 1+K,\\\label{up}
\p_p(v_0-v_1)\leq& 0,\quad \p_p (v_0-v_{-1})\geq 0.\\\label{u1-1}
v_0-v_1 &+v_0-v_{-1}\geq0.
\end{align}
\end{theorem}
\begin{proof}
Let $(v^\ep_0,v^\ep_1,v^\ep_{-1})$ be given in \propref{lem:ve}.
Fix $r\in(0,1/2)$, apply $W^{2,1}_p$ interior estimate (see \cite{Li96}) to \eqref{ve_eq}, we have for any $q>3$,
\[
|v^\ep_0|_{W^{2,1}_q(\Omega_r)},\;|v^\ep_1|_{W^{2,1}_q(\Omega_r)},\;|v^\ep_{-1}|_{W^{2,1}_q(\Omega_r)} \leq C_r,
\]
where $C_r$ is independent of $\ep$. For any $r\in(0,1/2)$, $q>3$ and $i=-1,0,1$, using Sobolev Embedding Theorem, i.e., $W^{2,1}_q(\Omega_r)\subseteq\subseteq C^{1+\al,\frac{1+\al}{2}}(\overline{\Omega}_r)$, with $\al=1-3/q$, there exists a subsequence of $v^\ep_i \in C^{1+\al,\frac{1+\al}{2}}(\Omega\bigcup\{t=0\})\bigcap W^{2,1}_{q,\mathrm{loc}}(\Omega\bigcup\{t=0\})$, which we still denote by $v^\ep_i$, such that
\[
v^\ep_i\longrightarrow v_i\quad\hbox{uniformly in}\; C^{1+\al,\frac{1+\al}{2}}(\overline{\Omega}_r),\quad \hbox{weekly in} \; W^{2,1}_q(\Omega_r).
\]
By \eqref{ve_eq}, we have
\[\p_t v^\ep_i-\LL v^\ep_i\geq0\quad\hbox{in}\; \Omega.\]
Letting $\ep\rt0$, we have
\begin{align}\label{Lvi}
\p_t v_i-\LL v_i\geq0\quad\hbox{in}\; \Omega.
\end{align}

Moreover, from \eqref{ue1}, \eqref{ue-1}, \eqref{uep} and \eqref{ue1ue-1}, we obtain \eqref{u1}, \eqref{u-1}, \eqref{up} and \eqref{u1-1}, respectively.

\par
Now, we prove the first variational inequality in \eqref{v_eq}. From \eqref{Lvi}, \eqref{u1} and \eqref{u-1}
we obtain
\[
\min\Big\{\p_t v_0-\LL v_0,\;v_0-v_1+(1+K),\;v_0-v_{-1}-(1-K)\Big\}\geq0.
\]
In the following, we come to prove the equality holds.

Suppose $v_0-v_1>-(1+K)+3\ep_{0}$ and $v_0-v_{-1}>(1-K)+3\ep_{0}$ hold at some point $(p,t)\in \Omega\bigcup\{t=0\}$ for some $\ep_{0}>0$. By the continuity of $v_i$ and the uniformly convergence of $v^\ep_i$ in $C^{1+\al,\frac{1+\al}{2}}(\Omega\bigcup\{t=0\})$, we have
\[
(v^\ep_0-v^\ep_1)(p,t)>-(1+K)+2\ep_{0},\quad (v^\ep_0-v^\ep_{-1})(p,t)>(1-K)+2\ep_{0},
\]
for small enough $\ep>0$. Thus by the first equation in \eqref{ve_eq},
\[
(\p_t v^\ep_0-\LL v^\ep_0)(p,t)=0.
\]
Let $\ep\rt0$ we get $(\p_t v_0-\LL v_0)(p,t)=0$.
Therefore, we proved $(v_0,v_1,v_{-1})$ satisfies the first variational inequality in \eqref{v_eq}.
The other two variational inequalities in \eqref{v_eq} can be proved similarly.
\par
The proof of uniqueness is standard, we omit the details.
\end{proof}

From now on, let $(v_0,v_1,v_{-1})$, given in \thmref{thm:v}, denote the unique solution of the VI system \eqref{v_eq}.

\section{Free boundaries and optimal trading strategies }\label{sec:freeboundary}
\setcounter{equation}{0}
In this section, we study the free boundaries for the VI system \eqref{v_eq}, and provide the optimal trading strategies for the stock trading problem.
\par
Define the trading regions ${\cal S}_{i,j}$, switching from position $i$ to position $j$, $i\neq j \in\{0,-1,1\}$, as \begin{align*}
{\cal S}_{0,1}&=\{(p,t)\in\Omega\;|\;v_0-v_1+1+K=0\},\\[2mm]
{\cal S}_{0,-1}&=\{(p,t)\in\Omega\;|\;v_0-v_{-1}-(1-K)=0\},\\[2mm]
{\cal S}_{1,0}&=\{(p,t)\in\Omega\;|\;v_1-v_0-(1-K)=0\},\\[2mm]
{\cal S}_{-1,0}&=\{(p,t)\in\Omega\;|\;v_{-1}-v_0+1+K=0\}.
\end{align*}
Define the free trading boundaries for $0\leq t<T$ as
\begin{align*}
p_{0,1}(t)&=\inf\{p\;|\;(p,t)\in{\cal S}_{0,1}\},\\[2mm]
p_{-1,0}(t)&=\inf\{p\;|\;(p,t)\in{\cal S}_{-1,0}\}, \\[2mm]
p_{0,-1}(t)&=\sup\{p\;|\;(p,t)\in{\cal S}_{0,-1}\}, \\[2mm]
p_{1,0}(t)&=\sup\{p\;|\;(p,t)\in{\cal S}_{1,0}\}.
\end{align*}
Since $\p_p(v_0-v_1)\leq0$, $\p_p(v_0-v_1)\geq0$ in $\Omega$, we see that
\begin{align*}
{\cal S}_{0,1}&=\{(p,t)\in\Omega\;|\;p\geq p_{0,1}(t)\},\\[2mm]
{\cal S}_{-1,0}&=\{(p,t)\in\Omega\;|\;p\geq p_{-1,0}(t)\},\\[2mm]
{\cal S}_{0,-1}&=\{(p,t)\in\Omega\;|\;p\leq p_{0,-1}(t)\},\\[2mm]
{\cal S}_{1,0}&=\{(p,t)\in\Omega\;|\;p\leq p_{1,0}(t)\}.
\end{align*}
Denote
\[
p_0:=\frac{\rho-\mu_2}{\mu_1-\mu_2}\in (0,1).
\]

The second main result of this paper is the completely characterization of the free boundaries defined above.
\begin{theorem}\label{freeboundaries}
We have
\begin{align*}
p_{0,-1}(t),\quad p_{1,0}(t),\quad p_{-1,0}(t),\quad p_{0,1}(t)\in C^{\infty}((0,T]).
\end{align*}
The free boundaries $p_{0,-1}(t)$ and $p_{-1,0}(t)$ are strictly increasing, and $p_{0,1}(t)$ and $p_{1,0}(t)$ are strictly decreasing; and
they have no-overlapping,
\begin{align*}
0\leq p_{0,-1}(t)< p_{1,0}(t)< p_0< p_{-1,0}(t)< p_{0,1}(t)\leq 1,
\end{align*}
with
\begin{align*}
p_{1,0}(0+)=p_0,\quad p_{-1,0}(0+)=p_0.
\end{align*}
Moreover, there exists $t_1\geq\frac{1}{\mu_1-\rho}\log\frac{1+K}{1-K}$ such that $p_{0,1}(t)=1$ if $t\leq t_1$;
and $t_0\geq\frac{1}{\rho-\mu_2}\log\frac{1+K}{1-K}$, such that $p_{0,-}(t)=0$ if $t\leq t_0$.
\end{theorem}

\thmref{freeboundaries} is an immediately consequence of Propositions \ref{thm:region} - \ref{thm:p_fenli} in the rest part of this section.
\figref{fig2} demonstrates the shapes and locations of the free boundaries.
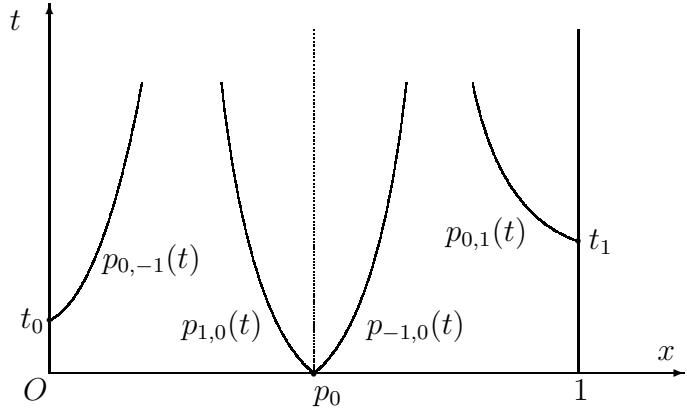
\begin{figure}[H]
\begin{center}
\begin{picture}(330,160)
\put(50,10){\vector(1,0){240}}
\put(50,10){\vector(0,1){140}}
\put(250,10){\line(0,1){130}}
\put(280,15){$x$}
\put(35,140){$t$}
\put(40,0){$O$}
\put(248,0){$1$}
\qbezier[100](150,10)(150,65)(150,140)

\put(150,10){\circle*{2}}\put(150,0){$p_0$}
\qbezier(150,10)(175,30)(185,120) \put(170,25){$p_{-1,0}(t)$}
\qbezier(150,10)(125,30)(115,120) \put(100,25){$p_{1,0}(t)$}
\qbezier(50,30)(70,40)(85,120) \put(70,50){$p_{0,-1}(t)$}\put(50,30){\circle*{2}}\put(39,27){$t_0$}
\qbezier(250,60)(220,70)(210,120) \put(200,60){$p_{0,1}(t)$}\put(250,60){\circle*{2}}\put(254,57){$t_1$}
\end{picture} \medskip
\caption{The free boundaries.} \label{fig2}\end{center}
\end{figure}

We summarize the optimal trading strategies for all the positions and status in Table \ref{table1}. \bigskip
\begin{table}[H]
\begin{center}
\begin{center}
\begin{tabular}{|p{4cm}<{\centering}|p{3cm}<{\centering}|p{3cm}<{\centering}|p{3cm}<{\centering}|}
\hline
status                         & in short position  & in flat position & in long position   \\ \hline
$0<p\leq p_{0,-1}(t)$          & do nothing & sell 1 share    & sell 2 share \\ \hline
$p_{0,-1}(t)<p\leq p_{1,0}(t)$ & do nothing & do nothing & sell 1 share\\ \hline
$p_{1,0}(t)< p<p_{-1,0}(t)$    & do nothing & do nothing & do nothing \\ \hline
$p_{-1,0}(t)\leq p<p_{0,1}(t)$ & buy 1 share& do nothing & do nothing \\ \hline
$p_{0,1}(t)\leq p<1$           & buy 2 share& buy 1 share& do nothing \\ \hline
\end{tabular}
\end{center}
\end{center}
\caption{The optimal strategy} \label{table1}
\end{table}

\begin{proposition}\label{thm:region}

We have
\begin{align}\label{subset1}
{\cal S}_{0,1}\subseteq{\cal S}_{-1,0}\subseteq \big[p_0,1\big)\times (0,T],
\end{align}
and
\begin{align}\label{subset2}
{\cal S}_{0,-1}\subseteq{\cal S}_{1,0}\subseteq \big(0,p_0]\times (0,T].
\end{align}
As a consequence,
\begin{align}\label{weizhi}
0\leq p_{0,-1}(t)\leq p_{1,0}(t)\leq p_0\leq p_{-1,0}(t)\leq p_{0,1}(t)\leq 1,\quad 0\leq t<T.
\end{align}
\end{proposition}
\begin{proof}
We only prove \eqref{subset1} as the proof of \eqref{subset2} is similar.
\par
Due to \eqref{u1-1} and \eqref{v_eq}, we have
\[
v_0-v_1+1+K\geq v_{-1}-v_0+1+K\geq 0.
\]
This implies the first part of \eqref{subset1}, namely, ${\cal S}_{0,1}\subseteq{\cal S}_{-1,0}$.

Now, we prove the second part of \eqref{subset1}, ${\cal S}_{-1,0}\subseteq \big[p_0,1\big)\times (0,T]$.
For any $(p,t)\in {\cal S}_{-1,0}$, we have $v_{-1}-v_0+(1+K)=0$. Thus by \eqref{v_eq},
\[
\p_t v_{-1}-\LL v_{-1}\geq0.
\]
If $(p,t)\in {\cal S}_{-1,0}\setminus{\cal S}_{0,1}$, i.e. $v_0-v_1+1+K>0$. Noticing $v_0-v_{-1}-(1-K)=2K>0$, we have by \eqref{v_eq},
\[
\p_t v_0-\LL v_0=0.
\]
So
\[
0\leq \Big(\p_t -\LL \Big)\Big(v_{-1}-v_0\Big)=\Big(\p_t -\LL \Big)\Big(-(1+K)\Big)=[(\mu_1-\mu_2)p+\mu_2-\rho](1+K)
\]
leading to
\begin{align}\label{p>}
p\geq \frac{\rho-\mu_2}{\mu_1-\mu_2}=p_{0}.
\end{align}
Otherwise, $(p,t)\in {\cal S}_{0,1}$, i.e. $v_0-v_1+1+K=0$, then $v_1-v_0-(1-K)=2K>0$. By \eqref{v_eq},
\[
\p_t v_1-\LL v_1=0.
\]
Note that $v_{-1}-v_1+2(1+K)=0$, thus
\[
0\leq \Big(\p_t -\LL \Big)\Big(v_{-1}-v_1\Big)=\Big(\p_t -\LL \Big)\Big(-2(1+K)\Big)=2[(\mu_1-\mu_2)p+\mu_2-\rho](1+K)
\]
which also implies \eqref{p>}.
So we proved the second part of \eqref{subset1}.
\end{proof}

The financial significance of this proposition is obvious. For instance, the first part of \eqref{subset1} means that if the trader wants to transfer from the present short position to the long position, he must first enter the flat position before going to the long position. The second part of \eqref{subset1} means if the (estimated) chance of the market status being bull, $p$, is too small (that is, less than the constant $p_0$), the trader should not buy the stock.

\begin{lemma}\label{thm:vt}
We have
\begin{align}\label{u1t}
\p_t(v_0-v_1)&\leq 0,\\
\p_t(v_0-v_{-1})&\leq 0. \label{u-1t}
\end{align}
\end{lemma}
\begin{proof}
Come back to $(u^\ep_1,u^\ep_{-1})$, the solution of \eqref{ue_eq}.
It only needs to prove that, for $i=1$, $-1$, and any fixed $\Delta t\in(0,T)$,
\begin{align*}
\wu^\ep_i(p,t):=u^\ep_i(p,t+\Delta t)\leq u^\ep_i(p,t)+2\ep e^{\la T} \quad\hbox{in}\; Q:=[\ep,1-\ep]\times[0,T-\Delta t],
\end{align*}
where $\la=\mu_1+\mu_2$ is a constant.
Define
\[
U_i=e^{-\la t}u^\ep_i.
\]
Our problem reduces to proving $\wU_i(p,t):=U_i(p,t+\Delta t)$ satisfies
\begin{align}\label{Dt}
e^{\la \Delta t}\wU_i(p,t)\leq U_i(p,t)+2\ep e^{\la (T-t)} \quad\hbox{in}\; Q.
\end{align}

We first prove \eqref{u1t}.
Now, we argue by contradiction. Suppose \eqref{Dt} was not true when $i=1$.
Denote by $(p^*,t^*)$ a maximum point of $e^{\la \Delta t}\wU_1-U_1-2\ep$ in $Q$, then we have
\[
\Big(e^{\la \Delta t}\wU_1-U_1-2\ep \Big)(p^*,t^*)>0,
\]
as well as
\begin{align}\label{wue}
\Big(\wu^\ep_1-u^\ep_1-2\ep\Big)(p^*,t^*)>0.
\end{align}
Since $\p_p\Big(e^{\la \Delta t}\wU_1-U_1-2\ep \Big)=0$ on $\{p=\ep\}$ and $\{p=1-\ep\}$, Hopf lemma implies $(p^*,t^*)$ will not lie on the two boundaries. And noting that
\[
\Big(e^{\la \Delta t}\wU_1-U_1-2\ep \Big)(p,0)=u^\ep_1(p,\Delta t)-[-(1-K)]-2\ep<0,
\]
we conclude that $(p^*,t^*)$  is inside the domain $Q$ or at the upper boundary of $Q$, so
\begin{align*}
\Big(e^{\la \Delta t}\wU_1-U_1\Big)(p^*,t^*)&>0, & \p_p\Big(e^{\la \Delta t}\wU_1-U_1\Big)(p^*,t^*)&=0,\\[2mm]
\p_{pp}\Big(e^{\la \Delta t}\wU_1-U_1\Big)(p^*,t^*)&\leq0,& \p_t\Big(e^{\la \Delta t}\wU_1-U_1\Big)(p^*,t^*)&\geq0.
\end{align*}
The above implies that
\begin{align}\label{wu_eq}
\Big(\p_t -\LL \Big)\Big(\wu^\ep_1-u^\ep_1\Big)\bigg|_{(p^*,t^*)}=e^{\la t}\Big(\p_t -\LL +\la\Big)\Big(e^{\la \Delta t}\wU^\ep_1-U^\ep_1\Big)\bigg|_{(p^*,t^*)}>0.
\end{align}

On the other hand, since both $u^\ep_1$ and $\wu^\ep_1$ satisfy the first equation in \eqref{ue_eq}, we have
\begin{align}\label{ue*}
\Big(\p_t -\LL \Big)(\wu^\ep_1-u^\ep_1)&=-\beta_\ep(\wu^\ep_1+1+K)+\beta_\ep(-\wu^\ep_1-(1-K))-\beta_\ep(\wu^\ep_{-1}-(1-K))\nonumber\\
&\quad\;+\beta_\ep(u^\ep_1+1+K)-\beta_\ep(-u^\ep_1-(1-K))+\beta_\ep(u^\ep_{-1}-(1-K)).
\end{align}
By \eqref{wue} and \eqref{ue1}, we get
\[
-u^\ep_1(p^*,t^*)-(1-K)>-\wu^\ep_1(p^*,t^*)-(1-K)+2\ep\geq2\ep.
\]
And by \eqref{ue1ue-1},
\[
u^\ep_{-1}(p^*,t^*)-(1-K)\geq-u^\ep_1(p^*,t^*)-(1-K)> 2\ep.
\]
Thus, by the definition of $\beta_\ep$,
\[
\beta_\ep(u^\ep_{-1}-(1-K))\Big|_{(p^*,t^*)}=\beta_\ep(-u^\ep_1-(1-K))\Big|_{(p^*,t^*)}=0.
\]
Moreover, due to $\beta_\ep'\geq 0$, $-\wu^\ep_1\leq \wu^\ep_{-1}$, and $u^\ep_1(p^*,t^*)<\wu^\ep_1(p^*,t^*)$,
we have
\[
\beta_\ep(-\wu^\ep_1-(1-K))-\beta_\ep(\wu^\ep_{-1}-(1-K))\Big|_{(p^*,t^*)}\leq0,\quad
\beta_\ep(u^\ep_1+1+K)-\beta_\ep(\wu^\ep_1+1+K)\Big|_{(p^*,t^*)}\leq0.
\]
Therefore, by \eqref{ue*} we have
\begin{align*}
\Big(\p_t -\LL \Big)\Big(\wu^\ep_1-u^\ep_1\Big)\bigg|_{(p^*,t^*)}\leq0,
\end{align*}
which contradicts \eqref{wu_eq}. Therefore, \eqref{Dt} is true.

Letting $\ep\rightarrow0$ in \eqref{Dt}, we get
\[
\Big(v_0-v_1\Big)(p,t+\Delta t)\leq \Big(v_0-v_1\Big)(p,t),
\]
which implies \eqref{u1t}. The proof of \eqref{u-1t} is similar.
\end{proof}

By \lemref{thm:vt}, we obtain
\begin{proposition}\label{prop:monotonefreeboundary}
The free boundaries $p_{0,-1}(t)$ and $p_{-1,0}(t)$ are increasing, and $p_{0,1}(t)$ and $p_{1,0}(t)$ are decreasing.
\end{proposition}
The financial significance of this proposition is clear. Because the investor must take the flat position at maturity, the investor is getting less likely to transfer from the flat position to the other two positions, and more likely to leave from
the short and long positions to the flat position.
\par
When $t$ is very small, we know the exact values of $p_{0,1}(t)$ and $p_{0,-1}(t)$.
\begin{proposition}\label{p1}
There exists $t_1\geq\frac{1}{\mu_1-\rho}\log\frac{1+K}{1-K}$, such that $p_{0,1}(t)=1$ if $t\leq t_1$; and, there exists $t_0\geq\frac{1}{\rho-\mu_2}\log\frac{1+K}{1-K}$, such that $p_{0,-1}(t)=0$ if $t\leq t_0$.
\end{proposition}
\begin{proof}
We will only prove the first assertion, the proof of the other one is similar. Note that $u_1=v_0-v_1$ satisfies
\begin{align}\label{u1_va}
\min\Big\{\p_t u_1-\LL u_1,\;u_1+(1+K)\Big\}=0 \quad \hbox{in}\quad [p_0,1]\times[0,T].
\end{align}
Especially, at the right boundary $p=1$, we have
\begin{align}\label{p_l}
\left\{
\begin{array}{ll}
\p_tu_1(1,t)-(\mu_1-\rho)u_1(1,t)=-\la_1\p_pu_1(1,t),\quad \hbox{if}\; u_1(1,t)>-(1+K);\\[2mm]
\p_tu_1(1,t)-(\mu_1-\rho)u_1(1,t)\geq-\la_1\p_pu_1(1,t),\quad \hbox{if}\; u_1(1,t)=-(1+K);\\[2mm]
u_1(1,0)=-(1-K).
\end{array}
\right.
\end{align}
Define
\begin{align*}
\left\{
\begin{array}{ll}
Z(t)=-(1-K)e^{(\mu_1-\rho)t},\quad \hbox{if}\quad t<\frac{1}{\mu_1-\rho}\log\frac{1+K}{1-K};\\[2mm]
Z(t)=-(1+K),\quad \hbox{if}\quad t\geq\frac{1}{\mu_1-\rho}\log\frac{1+K}{1-K}.
\end{array}
\right.
\end{align*}
Then it satisfies
\begin{align*}
\left\{
\begin{array}{ll}
\p_tZ(t)-(\mu_1-\rho)Z(t)= 0,\quad \hbox{if}\quad Z(t)>-(1+K);\\[2mm]
\p_tZ(t)-(\mu_1-\rho)Z(t)\geq 0,\quad \hbox{if}\quad Z(t)=-(1+K);\\[2mm]
Z(0)=-(1-K).
\end{array}
\right.
\end{align*}
Since $\p_pu_1(1,t)\leq 0$, we see that $Z(t)$ is a sub-solution of \eqref{p_l}.
Therefore, $ u_1(1,t)\geq Z(t)$. In particular, when $t<\frac{1}{\mu_1-\rho}\log\frac{1+K}{1-K}$, we have $u_1(1,t)\geq Z(t)>-(1+K)$, which implies $p_{0,1}(t)=1$.
\end{proof}

\begin{proposition}\label{thm:p_qidian}
The free boundaries $p_{1,0}(t)$ and $p_{-1,0}(t)$ are continuous. Moreover, both their initial points are $p_0$, i.e.,
\begin{align*}
p_{1,0}(0+)=p_0,\quad p_{-1,0}(0+)=p_0.
\end{align*}
\end{proposition}
\begin{proof}
We first prove the continuity property of $p_{1,0}(t)$.
By contrary, suppose $p_{1,0}(t)$ is discontinuous at a point $t_0\in(0,T]$, then by the decreasing property of $p_{1,0}(t)$, we have
\begin{equation*}
p_{1,0}(t_0+)<p_{1,0}(t_0-)\leq p_0.
\end{equation*}
Then $u_1=v_0-v_1$ satisfies
\[
u_1(p,t_0)=-(1-K),\quad \forall\; p\in (p_{1,0}(t_0+),p_{1,0}(t_0-)).
\]
Using \eqref{v_eq}, we have
\[
\Big(\p_t u_1-\LL u_1\Big)(p,t_0)=0,\quad p\in (p_{1,0}(t_0+),p_{1,0}(t_0-)).
\]
So $\p_t u_1(p,t_0)=\LL u_1(p,t_0)=[(\mu_1-\mu_2)p+\mu_2-\rho][-(1-K)]>0$  for $p\in (p_{1,0}(t_0+),p_{1,0}(t_0-))\subset[0,p_0]$, which contradicts \eqref{u1t}.
Similarly, we can prove $p_{1,0}(0+)=p_0$ and the corresponding properties of $p_{-1,0}(t)$.
\end{proof}

This result is some surprising to us. It tells us when the time is close to the maturity (that is, $t\leq t_{1}$), the investor should \emph{not} transfer from the flat position to the other positions at all. Such phenomenon has also been observed by Dai, Xu and Zhou (2010).

\begin{proposition}\label{thm:p_qidian}
The free boundaries $p_{0,1}(t)$ and $p_{0,-1}(t)$ are continuous.
\end{proposition}
\begin{proof}
We just prove the continuity property of $p_{0,1}(t)$.
If not, suppose $t_0$ was the discontinuous point, by the decreasing property of $p_{0,1}(t)$, we would have
\begin{equation*}
p_{0,1}(t_0+)<p_{0,1}(t_0-),
\end{equation*}
thus, $u_1=v_0-v_1$ satisfies
\begin{align}\label{1+K}
u_1(p,t_0)=-(1+K),\quad \forall p\in (p_{0,1}(t_0+),p_{0,1}(t_0-)).
\end{align}
Let ${\cal D}:=[p_{0,1}(t_0+),p_{0,1}(t_0-)]\times[t_0-\ep,t_0]$.
Using \eqref{v_eq}, we have
\[
\Big(\p_t u_1-\LL u_1\Big)(p,t)=0,\quad (p,t)\in {\cal D}.
\]
Note that $(p_{0,1}(t_0+),t_0)$ is the minimal point of $u_1$ in ${\cal D}$, by the strong maximum principle, \eqref{1+K} is impossible.
\end{proof}

Base on \eqref{weizhi} and the monotonicity of the four free boundaries, we further have the no-overlapping property.
\begin{proposition}\label{thm:p_fenli}
For the four free boundaries, we have,
\begin{align*}
0\leq p_{0,-1}(t)< p_{1,0}(t)< p_0< p_{-1,0}(t)< p_{0,1}(t)\leq 1, \quad t\in(0,T].
\end{align*}
Moreover, they are strictly monotone.
\end{proposition}
\begin{proof}
We just prove the first strict inequality. Suppose there exits $t_0\in(0,T]$ such that $p_{0,-1}(t_0)=p_{1,0}(t_0)$. Denote $c=p_{0,-1}(t_0)=p_{1,0}(t_0)$. Since $p_{0,-1}(t)$ is increasing, $p_{1,0}(t)$ is decreasing and $p_{0,-1}(t)\leq p_{1,0}(t)$, we have
\[
p_{0,-1}(t)=p_{1,0}(t)\equiv c,\quad\forall\; t\in [t_0,T].
\]
Thus $u_1=v_0-v_1$ satisfies
\begin{align}\label{partialtp}
u_1(c,t)=-(1-K),\quad \p_p u_1(c,t)=0,\quad \p_t u_1(c,t)=0,\quad\p_{tp} u_1(c,t)=0,\quad\forall\; t\in [t_0,T].
\end{align}
By \eqref{u1t}, $\p_t u_1\leq 0$, the above means $\p_t u_1$ gets its maximum value $0$ at $(c,t)$ for any $t\in [t_0,T]$.
\par
On the other hand, note that
\[
\Big(\p_t-\LL \Big)u_1=0,\quad\forall\; (p,t)\in [c,p_0]\times[t_0,T],
\]
differential it w.r.t. $t$ we have
\[
\Big(\p_t-\LL \Big)\p_t u_1=0,\quad\forall\; (p,t)\in [c,p_0]\times[t_0,T].
\]
Since $\p_t u_1$ gets its maximum value at $(c,t)$ for any $t\in [t_0,T]$, by the Hopf lemma, we have
\[
\p_{tp} u_1(c,t)<0,
\]
which contradicts to \eqref{partialtp}.
The remaining conclusions can be proved similarly.
\end{proof}

\begin{proposition}\label{thm:p_fenli}
We have
\begin{align*}
p_{0,-1}(t),\quad p_{1,0}(t),\quad p_{-1,0}(t), \quad p_{0,1}(t)\in C^{\infty}((0,T]).
\end{align*}
\end{proposition}
\begin{proof}
The proof of the smoothness is technical, we refer interested reads to \cite{DZ10,DX10,Fr75}
\end{proof}

\newpage
\section*{Appendix}

\section*{Financial background}\setcounter{equation}{0}
Let $(\Omega, \cF, (\cF_{t})_{t\geq 0}, \mathbb{P})$ be a fixed filtered complete probability space on which is defined a standard one-dimensional Brownian motion $W$. It represents the financial market.
A stock is given in the market and its price process $S=(S_{t})_{t\geq 0}$ satisfies the stochastic differential equation (SDE)
\[\dd S_{t}=\mu(I_{t}) S_{t}\dt+\sigma S_{t}\dd W_{t},\]
where the volatility $\sigma>0$ is a known constant, but the market status process $I_{t}$ and the noice process $W_{t}$ are unobservable to the investor.
It is also assumed that $\cF_t$ is equal to $\sigma\{S(s):0\leq s\leq t\}$ augmented by all the
$\mathbb{P}$-null sets and $\cF_T\subseteq \cF$.
\par
To model drift uncertainty, we assume $I_s\in \{1,2\}$ is a two-state Markov
chain. At each time $s$, $I_s=1$ indicates a bull market status and $I_s=2$ a bear market status.
Suppose the transition probabilities are given by
\begin{align*}
\P(I_{s+h}=j\;|\;I_s=i)=\left\{
\begin{array}{ll}
\la_i h+o(h),& j\neq i,\\[2mm]
1-\la_i h+o(h),& j=i,
\end{array}
\right.\quad i,j=1,2.
\end{align*}
We assume $I$ and $W$ are independent processes and $\mu_{1}> \mu_{2}$ where $\mu_i=\mu(i)$, $i = 1, 2.$
If $\mu_{1}=\mu_{2}$, then the market status uncertainty disappears and our financial market becomes the classical Black-Scholes market.
\par
The investor is allowed to trade the stock at any time. But, at each time, the investor can only take one of the three positions: $\{-1, 0, 1\}$, where $i=0$ corresponds to a flat position (no stock holding), $i=1$ to a long position (holding one share of the stock), while $i=-1$ to a short position (short sale one share of the stock).
\par
The investor's trading strategy beginning at time $t$ is modeled by a sequence $(\tau_n, \iota_n)_{n\in \Z_+}$, where $t=\tau_0<\tau_1<\cdots$ is a strictly increasing sequence of $(\cF_{s})_{s\geq t}$-stopping times, and $\iota_n$ valued in $\{-1, 0, 1\}$, requiring ${\cal F}_{\tau_n}$ measurable, represents the position regime decided at $\tau_n$ until the next trading time. By misuse of notation, we denote by $\al^t_s$ the value of the regime at any time $s$ (begin with time $t$), namely,
\[
\al^t_s = \iota_0 1_{\{t\leq s <\tau_1\}} +\sum_{n=1}^{\infty}\iota_n 1_{\{\tau_n\leq s <\tau_{n+1}\}},\quad s\geq t.
\]
Since $(\tau_n, \iota_n)_{n\in \Z_+}$ and $\al^t$ are one-to-one, we will not distinguish them and call them switching controls starting from time $t$.

Let $0<K<1$ denote the percentage of transaction fee.
We denote by $g_{i,j}(S)$ the trading gain when switching from a position $i$ to $j$, $i, j\in\{-1, 0, 1\}$, $j\neq i$, for a current stock price $S$. The switching gain functions are given by:
\begin{align*}
&g_{0,1}(S)=g_{-1,0}(S)=-S(1+K),\\
&g_{0,-1}(S)=g_{1,0}(S)=S(1-K).
\end{align*}
For integrity, set
\begin{align*}
&g_{0,0}(S)=g_{1,1}(S)=g_{-1,-1}(S)=0,\\
&g_{-1,1}(S)=g_{-1,0}(S)+g_{0,1}(S)=-2S(1+K),\\
&g_{1,-1}(S)=g_{1,0}(S)+g_{0,-1}(S)=2S(1-K).
\end{align*}
By misuse of notation, we also set $g(S,i,j)=g_{i,j}(S)$.

At initial time $t$, given the stock price $S_t = S$, the reward functions of the decision sequences for a switching control $\al$ starting from time $t$, is given by
\[
J(S,t,\al) = \E \bigg[\sum_{n=1}^{\infty} e^{-\rho (\tau_n-t)} g(S_{\tau_n},\al_{\tau_n^-},\al_{\tau_n})1_{\{\tau_n\leq T\}}+e^{-\rho (T-t)} g(S_T,\al_T,0)\bigg|S_t = S \bigg],
\]
where $\rho > 0$ is the discount factor satisfying
\[
\mu_2<\rho<\mu_1.
\]
If $\rho$ is out of the above range, the problem admits trivial solutions only.
The last term in the reword function means that at the maturity time $T$ the net position must be flat.
\par
Note that only the stock price $S_s$ is observable to the investor at time $s$. The market status $I_s$ is not directly observable. Thus, it is necessary to convert the problem into a completely
observable one. One way to accomplish this is to use the Wonham filter \cite{Wo65}.
Let $p_s = \P(I_s =1\;|\;\cF_{s})$. Then we can show (see \cite{Wo65}) that $p_s$ satisfies
the following SDE
\[
\d p_s = [-(\la_1 +\la_2) p_s + \la_2] \d s + (\mu_1-\mu_2)p_s(1-p_s)\sigma\d \nu_s,
\]
where $\nu_s$, given by
\[
\d \nu_s =\frac{ \d \log(S_s) + [(\mu_1+\mu_2)p_s + \mu_2-\sigma^2/2]\d s}{\sigma},
\]
is called the innovation process and is a standard Brownian motion under the filtration $(\cF_{t})_{t\geq 0}$ (see, e.g., \cite{Ok03}).

Given $S_t = S$ and $p_t = p$, the investor's problem is to choose a switching control $(\al^t_s)_{s\geq t}$ starting from $t$, to maximize the reward function
\[
J(S,p,t,\al^t)
\]
subject to
\begin{align*}
&\d S_s = S_s [(\mu_1-\mu_2)p_s + \mu_2] \d s + S_s\sigma\d \nu_s,\quad S_t = S,\\
&\d p_s = [-(\la_1 + \la_2)p_s + \la_2] \d s + (\mu_1-\mu_2)p_s(1-p_s)\sigma\d \nu_s, \quad p_t = p.
\end{align*}
Indeed, this new problem is a completely observable one, because the conditional probabilities
can be obtained using the stock price up to time $s$.
\par
Let $V_i(S, p, t)$ denote the value function with the state $(S,t,p)$ and net positions $i \in\{-1,0,1\}$ at time $t$. That is,
\[
V_i(S, p, t)=\sup\limits_{\al\in{\cal A}^i_t}J(S, p, t,\al^t),
\]
where ${\cal A}^i_t$ denotes the set of switching controls $\al^t $ starting from $t$ with initial position $\iota_0 = i$.

Define an operator
\begin{align*}
\TT &=\frac{1}{2}\Big(\frac{(\mu_1-\mu_2)p(1-p)}{\si}\Big)^2\p_{pp} +\frac{1}{2}\si^2S^2\p_{SS}+S(\mu_1-\mu_2)p(1-p)\p_{Sp}\\[2mm]
&\quad\;+[-(\la_1+\la_2)p+\la_2]\p_p+S[(\mu_1-\mu_2)p+\mu_2]\p_S-\rho.
\end{align*}
The HJB equation associated with our optimal stopping time problem can be given formally as follows
\begin{align*}
\left\{
\begin{array}{ll}
\min\Big\{-\p_t V_0-\TT V_0,\;V_0-V_1+S(1+K),\;V_0-V_{-1}-S(1-K)\Big\}=0,\\[2mm]
\min\Big\{-\p_t V_1-\TT V_1,\;V_1-V_0-S(1-K)\Big\}=0,\\[2mm]
\min\Big\{-\p_t V_{-1}-\TT V_{-1},\;V_{-1}-V_0+S(1+K)\Big\}=0,
\end{array}
\right.
\quad
\begin{array}{ll}0<p<1,\\[2mm]0\leq t<T,\end{array}
\end{align*}
with the terminal conditions
\begin{align*}
\left\{
\begin{array}{ll}
V_0(p,T)=0,\\[2mm]
V_1(p,T)=S(1-K),\\[2mm]
V_{-1}(p,T)=-S(1+K),
\end{array}
\right.
\quad 0<p<1.
\end{align*}

It is easy to see from the model that the value functions $V_i(S,p,t),\;i=-1,0,1$ are affine in $S$. This motivates us to adopt the following transformation $\widehat{V}_i(p, t) = V_i(S,p,t)/S$.
Further, for convenience of discussion, In order to transform the terminal value problem to an initial value problem, let $v_i(p, t) = \widehat{V}_i(p, T-t)$, then $v_i$ satisfies the following variational inequality system
\begin{align*}
\left\{
\begin{array}{ll}
\min\Big\{\p_t v_0-\LL v_0,\;v_0-v_1+(1+K),\;v_0-v_{-1}-(1-K)\Big\}=0,\\[2mm]
\min\Big\{\p_t v_1-\LL v_1,\;v_1-v_0-(1-K)\Big\}=0,\\[2mm]
\min\Big\{\p_t v_{-1}-\LL v_{-1},\;v_{-1}-v_0+(1+K)\Big\}=0,
\end{array}
\right.
\; (p,t)\in\Omega,
\end{align*}
with initial conditions
\begin{align*}
\left\{
\begin{array}{ll}
v_0(p,0)=0,\\[2mm]
v_1(p,0)=1-K,\\[2mm]
v_{-1}(p,0)=-(1+K),
\end{array}
\right.
\quad 0<p<1,
\end{align*}
where $\Omega=(0,1)\times[0,T)$ and
\begin{align*}
\LL &=\frac{1}{2}\Big(\frac{(\mu_1-\mu_2)p(1-p)}{\si}\Big)^2\p_{pp}\\[2mm]
&\quad\;+[-(\la_1+\la_2)p+\la_2+(\mu_1-\mu_2)p(1-p)]\p_p+(\mu_1-\mu_2)p+\mu_2-\rho.
\end{align*}
This is the variational inequality system \eqref{v_eq} that is investigated in this paper.

\newpage

\bibliographystyle{plainnat}


\end{document}